\documentclass[12pt]{amsart}
\usepackage{amscd}
\usepackage{amssymb}
\usepackage{a4wide}
\usepackage{amstext}
\usepackage{amsthm}
\usepackage{mathrsfs}
\usepackage{relsize}
\usepackage[final]{hyperref}
\hypersetup{unicode= false, colorlinks=true, linkcolor=blue,
anchorcolor=blue, citecolor=green, filecolor=red, menucolor=blue,
pagecolor=blue, urlcolor=blue} 
\numberwithin{equation}{section}

\newcommand{\CC}{\mathbb {C}}

\newcommand{\RR}{\mathbb{R}}

 \DeclareMathOperator{\dist}{dist}

\renewcommand{\phi}{\varphi}
\newcommand{\rea}{{\rm Re}\,}
\newcommand{\ima}{{\rm Im}\,}

\newcommand{\vep}{\varepsilon}
\newcommand{\co}{\mathbb{C}}

\newcommand{\ZZ}{\mathcal{Z}}
\newcommand{\cp}{\mathbb{C^+}}
\newcommand{\cm}{\mathbb{C^-}}

\newtheorem{Thm}{Theorem}[section]
\newtheorem{theorem}[Thm]{Theorem}
\newtheorem{lemma}[Thm]{Lemma}
\newtheorem{proposition}[Thm]{Proposition}
\newtheorem{corollary}[Thm]{Corollary}
\newtheorem{remark}[Thm]{Remark}
\newtheorem{example}[Thm]{Example}

\begin{document}
\sloppy
\title[Krein-type theorems and Cauchy--de Branges spaces]
{Krein-type theorems and ordered structure \\ for Cauchy--de Branges spaces}
\author{Evgeny Abakumov, Anton Baranov, Yurii Belov}

\address{Evgeny Abakumov,
\newline  University Paris-Est,  LAMA (UMR CNRS 8050),  Marne-la-Vall\'ee, France
\newline {\tt evgueni.abakoumov@u-pem.fr}
\smallskip
\newline \phantom{x}\,\, Anton Baranov,
\newline Department of Mathematics and Mechanics, St.~Petersburg State University, St.~Petersburg, Russia,
\newline National Research University Higher School of Economics, St.~Petersburg, Russia,
\newline {\tt anton.d.baranov@gmail.com}
\smallskip
\newline \phantom{x}\,\, Yurii Belov,
\newline  St.~Petersburg State University, St. Petersburg, Russia,
\newline {\tt j\_b\_juri\_belov@mail.ru}
}
\thanks{ The work was supported by the joint grant of Russian Foundation for Basic Research 
(project 17-51-150005-NCNI-a) and CNRS, France 
(project PRC CNRS/RFBR 2017-2019 ``oyaux reproduisants dans
des espaces de Hilbert de fonctions analytiques''), 
and by the Ministry of Education and Science of the
Russian Federation (project 1.3843.2017). }

\maketitle

\begin{abstract}
We extend some results of M.G. Krein to the class of entire functions 
which can be represented as ratios of discrete Cauchy transforms in the plane.
As an application we obtain new versions of de Branges' Ordering Theorem 
for nearly invariant subspaces in a class of Hilbert spaces of 
entire functions. Examples illustrating sharpness of the obtained 
results are given. 
\end{abstract}

\section{Introduction}
\label{int}

\subsection{Krein's theorem.}
M.G. Krein's theorem about the Cartwright class functions 
plays a seminal role in entire function theory and its 
applications to spectral theory of linear operators.
Recall that an entire function $F$ is said to be of {\it Cartwright class}
if it is of finite exponential type and
$$
\int_{\RR} \frac{\log^+|F(x)|}{1+x^2}dx <\infty.
$$
Krein's theorem can be stated 
as follows (for the necessary background see Section \ref{prel}): 
\medskip
\\
{\bf Theorem (M.G. Krein).} {\it Let $F$ be an entire function. If $F$ 
is a function of bounded type both in the upper half-plane $\mathbb{C}^+$
and the lower half-plane $\mathbb{C}^-$, then $F$ is a function of 
Cartwright class. In particular, $F$ is of finite exponential type
and its type is equal to $\max ({\rm mt}_+(F), {\rm mt}_-(F))$, where
${\rm mt}_+(F)$ and ${\rm mt}_-(F)$ denote the mean type of $F$ in 
$\cp$ and in $\cm$, respectively.  }
\medskip

For different approaches to this result see 
\cite[Part II, Chapter 1]{hj} or \cite[Lecture 16]{lev};
its applications to the spectral theory of non-dissipative operators
can be found, e.g., in ~\cite[Section IV.8]{Gohb_Krein}.

A typical situation when Krein's theorem is applicable is when $F$ can be represented
as a ratio of two Cauchy transforms of some discrete measures on $\RR$. Namely, 
for $T = \{t_n\}_{n=1}^\infty \subset \RR$ and $a= \{a_n\}_{n=1}^\infty \in \ell^1$ 
consider the discrete {\it Cauchy transform}
$$
\mathcal{C}_a(z)  = \sum_n \frac{a_n}{z- t_n}.
$$
Condition $a\in \ell^1$ can be relaxed. Assume that
$t_n \ne 0$ and, for some $k\in \mathbb{N}$, 
$\sum_n |t_n|^{-k-1}|a_n| <\infty$. Then 
we define the {\it regularized Cauchy transform} as
\begin{equation}
\label{reg}
\mathcal{C}_{a, k} (z)  = \sum_n a_n\bigg(\frac{1}{z- t_n} +\frac{1}{t_n} +
+\dots + \frac{z^{k-1}}{t_n^k}\bigg) = 
z^k \sum_n \frac{a_n}{t_n^k (z-t_n)}
\end{equation}
(we do not need to regularize the Cauchy kernel at infinity when $k=0$).
The functions of the form $\mathcal{C}_{a, k}$ are of bounded type 
in $\cp$ and in $\cm$. Therefore, if an entire function $F$ can be represented as 
$\mathcal{C}_{a, k}/\mathcal{C}_{b, m}$, 
then $F$ is of finite exponential type. 

A special case of the above statement is the following theorem, 
also due to Krein
(see \cite[Theorem 4]{krein47} or \cite[Lecture 16]{lev}): {\it Assume
that $F$ is an entire function, which is real on $\RR$,
with simple real zeros $t_n \ne 0$ and such that,
for some integer $k\ge 0$, we have
$$
\sum_n \frac{1}{|t_n|^{k+1} |F'(t_n)|}<\infty
$$
and 
\begin{equation}
\label{krein}
\frac{1}{F(z)} = R(z) + \sum_n
\frac{1}{F'(t_n)} \cdot \bigg(\frac{1}{z-t_n} +\frac{1}{t_n}+
\cdots + \frac{z^{k-1}}{t_n^k} \bigg), 
\end{equation}
where $R$ is some polynomial. 
Then $F$ is of Cartwright class.}
Krein \cite[Theorem 5]{krein47} showed also that the condition 
$t_n\in\mathbb{R}$ can be relaxed to the Blaschke condition
$\sum_n |t_n|^{-2} |\ima t_n| <\infty$.
Some further refinements of this result are due to
A.G.~Bakan and V.B.~Sherstyukov (see, e.g., \cite{shers}
and references therein).
\medskip

One of the goals of this paper is to extend the above results to the case 
of Cauchy transforms of measures which are supported by some discrete 
set $\{t_n\}$ in $\CC$ where $t_n$ {\it are no longer real}. 
In what follows we assume that $T=\{t_n\}\subset \CC$, $t_n$ 
are pairwise distinct, and $|t_n|\to \infty$
as $n\to\infty$. To simplify the formulas we assume, as above,
that $0\notin T$ 
(if $0\in T$ then an obvious modification of the formulas is required, but all 
results remain true). 
\medskip
\\
{\bf Question.} {\it Let $F$ be an entire function such that 
$F = \mathcal{C}_{a, k}/\mathcal{C}_{b, m}$ for some $a,b, k, m$. 
Under what conditions on $T$ can we conclude that $F$ is a function 
of finite exponential type?}
\medskip

We find several conditions on $T$ ensuring that this 
(and even more) is true. We also provide some examples which 
show the sharpness of these conditions.


\subsection{The spaces of Cauchy transforms} 
One motivation for the study of the above question is related 
to spectral theory of rank one perturbations of compact normal operators. 
Recently, D.V.~Yakubovich and the second author \cite{by}
applied a functional model 
to the study of rank one perturbations of compact selfadjoint operators.
Using this model, a number of results was obtained about completeness 
and spectral synthesis for such perturbations. 
The functional model in question acts in the so-called de Branges spaces  
of entire functions which can be identified with the spaces of Cauchy 
transforms of discrete measures supported by $\RR$. 
To extend the results of \cite{by} to the case 
of rank one perturbations of {\it normal} (non-selfadjoint) compact operators one needs to 
have analogues of Krein's theorems for the case of non-real $t_n$. 
A similar functional model for perturbations
of normal operators was constructed in \cite{by1}; 
it also acts in some space of discrete Cauchy transforms, 
which we introduce now. 

Let $T = \{t_n\}_{n=1}^\infty$ be a set as above 
  and let 
$\mu=\sum_n\mu_n\delta_{t_n}$ be a positive measure such that
$\sum_n \frac{\mu_n}{|t_n|^2 +1}<\infty$.
Also let $A$ be an entire function which has only simple zeros and whose zero set 
$\ZZ_A$ coincides with $T$. With any such $T$, $A$ and $\mu$
we associate the space $\mathcal{H}(T,A,\mu)$ of entire functions,
$$
\mathcal{H}(T,A,\mu):=\biggl{\{}f:f(z)= A(z)\sum_n\frac{a_n\mu^{1/2}_n}{z-t_n},
\quad a = \{a_n\}\in\ell^2\biggr{\}}
$$
equipped with the norm
$\|f\|_{\mathcal{H}(T,A,\mu)}:=\|a\|_{\ell^2}$. In what follows, the spaces
${\mathcal{H}(T,A,\mu)}$ will be called {\it Cauchy--de Branges spaces}.

The spaces $\mathcal{H}(T,A,\mu)$ were introduced in full generality
by Yu. Belov, T. Mengestie
and K. Seip \cite{bms}. Essentially, they are spaces of Cauchy transforms. 
We multiply them by the entire function $A$   to get rid of poles and make the elements entire, but 
essentially the space does not depend on the choice of $A$. The spaces with the same 
$T$, $\mu$ and different $A$-s are isomorphic. 
In what follows we will always assume that $T$ has a finite 
convergence exponent (i.e., $\sum_n |t_n|^{-K} <\infty$ for some $K>0$)
and $A$ is some canonical product of the corresponding order. 
We call the pair $(T, \mu)$ the {\it spectral data} for $\mathcal{H}(T,A,\mu)$. 

It is noted in \cite{bms} that each space $\mathcal{H}(T,A,\mu)$ 
is a reproducing kernel Hilbert space and, moreover, if $\mathcal{H}$
is a reproducing kernel Hilbert space of entire functions such that
\medskip
\begin{enumerate} 
\item [(i)]
$\mathcal{H}$ has the {\it division property}, that is, $\frac{f(z)}{z-w} 
\in \mathcal{H}$ whenever $f\in\mathcal{H}$ and $f(w) = 0$,
\medskip
\item [(ii)]
there exists a Riesz basis of reproducing kernels in $\mathcal{H}$,
\end{enumerate}
\medskip
then $\mathcal{H} = \mathcal{H}(T,A,\mu)$ (as sets with equivalence of norms) 
for some choice of the parameters. 

Note that the functions $\overline{A'(t_n)} \mu_n \cdot \frac{A(z)}{z-t_n}$
form an orthogonal basis in $\mathcal{H}(T,A,\mu)$ and are 
the reproducing kernels at the points $t_n$.  

In the case when $T\subset \RR$ and $A$ is real on $\RR$, 
the space $\mathcal{H}(T,A,\mu)$ is a de Branges space. 
De Branges spaces' theory is a deep and important field 
which has numerous applications to operator theory, spectral theory of differential 
operators and even to number theory. For the basics of de Branges theory we refer to 
de Branges' monograph \cite{br} and to \cite{rom}; 
some further results and applications can be found in \cite{abb, by, lag, mp, os, rem}. 

Since the spaces $\mathcal{H}(T,A,\mu)$ are defined in terms of Cauchy
transforms and also are a generalization of de Branges spaces, 
the term {\it a Cauchy--de Branges space} seems to be appropriate.

We believe it is a noteworthy problem to extend certain aspects of de Branges theory to 
the more general setting of Cauchy--de Branges spaces.
One of the most striking features of de Branges spaces is 
the ordered structure 
of their subspaces which are themselves de Branges spaces:
{\it If $\mathcal{H}_1$ and $\mathcal{H}_2$ are two de Branges subspaces 
of a de Branges space $\mathcal{H}$, then either 
$\mathcal{H}_1 \subset \mathcal{H}_2$ or
$\mathcal{H}_2 \subset \mathcal{H}_1$} \cite[Theorem 35]{br}. 
Here we study this problem for Cauchy--de Branges spaces and,
as an application of the Krein-type theorems obtained in the first part of the paper,
establish the ordering property for a class of these spaces. 
 

\subsection{Main results} 
We will develop the Krein-type theory for measures with non-real supports in the 
following three cases:
\medskip
\begin{enumerate} 
\item [(i)] $\mathbf{Z:}$ $T$ is the zero set of some entire function 
of zero exponential type; 
\smallskip
\item [(ii)]
$\mathbf{\Pi:}$ $T$ lies in some strip and has finite convergence exponent;
\smallskip
\item [(iii)] $\mathbf{A_\gamma}:$  $T$ lies in some angle of size $\pi\gamma$, $0<\gamma<1$,
and the convergence exponent of $T$ is strictly less than  $\gamma^{-1}$.
\end{enumerate}

\begin{theorem}
\label{main1}
Let $T$ satisfy one of the conditions $\mathbf{Z}$, $\mathbf{\Pi}$ 
or $\mathbf{A_\gamma}$. 
Assume that $F$ is an entire function such that $F = 
\mathcal{C}_{a, k}/\mathcal{C}_{b, m}$, 
where 
$\mathcal{C}_{a, k}$ and $\mathcal{C}_{b, m}$ are regularized 
Cauchy transforms with poles in $T$ defined in \eqref{reg}. 
Then $F$ is a function of finite exponential type.

In cases $\mathbf{Z}$ and $\mathbf{A_\gamma}$, the function 
$F$ is of zero exponential type. 
In cases $\mathbf{\Pi}$ and $\mathbf{A_\gamma}$, 
$F$ is of Cartwright class with respect to some line. 
\end{theorem}

As a corollary of Theorem \ref{main1} 
we see that if a finite order function $F$ with zeros 
in a strip or a function of order less than $\gamma^{-1}$ with zeros in the angle 
of size $\pi\gamma$ admits the representation \eqref{krein}, then $F$ is a function 
of exponential type. The first of these observations was proved in 
\cite{shers} where Krein-type theorems for functions with zeros in a strip were studied.  

The relation between the size of the angle and the order 
in the case $\mathbf{A_\gamma}$ is optimal.

\begin{theorem}
\label{count_a} 
For any $\gamma \in (0,1)$ there exists an entire function $F$  of order precisely $\gamma^{-1}$  
such that all its zeros $t_n$ are simple, lie in an angle of size $\pi\gamma$
and 
$$
\sum_n \frac{1}{|F'(t_n)|} <\infty, \qquad  
\frac{1}{F(z)} = \sum_n \frac{1}{F'(t_n) (z-t_n)}.
$$
\end{theorem}
\medskip

We use Theorem \ref{main1} to establish Ordering Theorems for 
the Cauchy--de Branges
spaces $\mathcal{H}(T,A,\mu)$. In fact, we consider a more general
and, in a sense, more natural class of subspaces: these are 
{\it nearly invariant} (or {\it division-invariant}) subspaces. 
A closed subspace $\mathcal{H}_0$ of a Cauchy--de Branges space $\mathcal{H}$
is said to be {\it nearly invariant} if there is $w_0\in \mathbb{C}$ 
such that $\frac{f(z)}{z-w_0} \in \mathcal{H}_0$ whenever $f\in \mathcal{H}_0$
and $f(w_0) =0$. It is known that this property is equivalent
to a stronger {\it division invariance property}: for any $w\in \mathbb{C}$
such that  there exists $f\in \mathcal{H}_0$ with $f(w)\ne 0$
($w$ is not a {\it common zero for $\mathcal{H}_0$}),
$$
f\in \mathcal{H}_0, \ \ f(w) = 0 \ \Longrightarrow \ 
\frac{f(z)}{z-w}\in \mathcal{H}_0.
$$
In the context of Hardy spaces in general domains the equivalence
of nearly invariance and division invariance is shown 
in \cite[Proposition 5.1]{ar}; a similar argument works 
for general spaces of analytic functions (see Proposition \ref{nea} below).

While the de Branges theory guarantees
a rich structure of de Branges subspaces in a de Branges space, it is not clear whether
there always exist many subspaces of $\mathcal{H}(T,A,\mu)$ which have a Riesz basis 
of reproducing kernels (i.e., are Cauchy--de Branges spaces themselves). 
However, there exist many nearly invariant
subspaces. A natural construction of a nearly invariant subspace 
is as follows. Given a function $G\in \mathcal{H}(T,A,\mu)$, 
consider the subspace of $\mathcal{H}(T,A,\mu)$ defined as 
\begin{equation}
\label{rs}
\mathcal{H}_G = \overline{\rm Span}\, \Big\{\frac{G}{z-\lambda}: G(\lambda) = 0 \Big\}.
\end{equation}
We can also define $\mathcal{H}_G$ in a slightly more general situation 
when $G$ possibly is not in $\mathcal{H}(T,A,\mu)$, but 
$\frac{G}{z-\lambda} \in \mathcal{H}(T,A,\mu)$ whenever $G(\lambda) = 0$.
It is easy to see that if $G$ has simple zeros then $\mathcal{H}_G$ 
is nearly invariant (and, thus, division-invariant). 
Clearly, any subspace $\mathcal{H}$ which is itself a Cauchy--de~Branges space 
is of the form \eqref{rs}
(indeed, if $\mathcal{H} = \mathcal{H}(T_1, A_1, \mu_1)$, 
then $\mathcal{H} = \mathcal{H}_{A_1}$). We do not know at present whether
any division-invariant subspace of $\mathcal{H}(T,A,\mu)$ is of the form 
$\mathcal{H}_G$.

We will consider mostly nearly invariant subspaces $\mathcal{H}_0$ 
{\it without common zeros}, that is, such that $\mathcal{Z}(\mathcal{H}_0) 
=\emptyset$, where $\mathcal{Z}(\mathcal{H}_0)=
\{w\in \co:\, f(w) = 0\ \text{for any}\ f\in \mathcal{H}_0\}$.
Note that subspaces of the form $\mathcal{H}_G$ 
have no common zeros when zeros of $G$ are simple. 

Now we formulate two theorems which extend  
de Branges' Ordering Theorem to Cauchy--de Branges spaces.

\begin{theorem}
\label{main2}
Let $T$ satisfy one of the conditions $\mathbf{Z}$ or $\mathbf{A_\gamma}$ 
and let $\mathcal{H}_1$ and $\mathcal{H}_2$ be two 
nearly invariant subspaces of $\mathcal{H}(T,A,\mu)$ 
without common zeros. 
Then either $\mathcal{H}_1 \subset \mathcal{H}_2$ or 
$\mathcal{H}_2 \subset \mathcal{H}_1$.
\end{theorem}

To state a similar result for the strip case $\mathbf{\Pi}$ 
we need to impose some conditions.
Otherwise the statement is no longer true even in the case of real zeros.

\begin{theorem}
\label{main3}       
Let $T \subset \{-h \le \ima z\le h\}$, $h>0$, 
and let $\mathcal{H}_1$ and $\mathcal{H}_2$ be two 
nearly invariant subspaces of $\mathcal{H}(T,A,\mu)$ without common zeros. 
Assume, moreover, that $\mathcal{H}_1$ and $\mathcal{H}_2$ 
are closed under the $*$-transform $f\mapsto f^*$, 
where $f^*(z) = \overline{f(\overline z)}$.
Then either $\mathcal{H}_1 \subset \mathcal{H}_2$ or 
$\mathcal{H}_2 \subset \mathcal{H}_1$.
\end{theorem}

In other words, in the above cases the set of all nearly invariant subspaces
without common zeros (and, in particular, the set of all Cauchy--de Branges subspaces)
of $\mathcal{H}(T,A,\mu)$ is totally ordered by inclusion.
However, without any restrictions on the growth or spectrum localization
the ordered structure for nearly invariant subspaces fails.
This is illustrated by the following
                         
\begin{theorem}
\label{main4}
There exists a space $\mathcal{H}(T,A,\mu)$ of order $2$ 
and two nearly invariant and $*$-closed subspaces
$\mathcal{H}_1$ and $\mathcal{H}_2$ without common zeros such that
neither $\mathcal{H}_1 \subset \mathcal{H}_2$ nor 
$\mathcal{H}_2 \subset \mathcal{H}_1$.
Moreover, these subspaces can be chosen to be of the form \eqref{rs}.
\end{theorem}

The paper is organized as follows. In Section \ref{prel} we discuss 
our main tools from function theory. Theorem \ref{main1} is proved 
in Section \ref{proofth1}, while in Section \ref{count0} counterexamples 
are given illustrating its sharpness.
Ordering theorems \ref{main3} and \ref{main4} are proved in 
Section \ref{proofth23}. Construction of two nearly invariant subspaces
which do not contain each other is presented in Section \ref{count}.
\bigskip


\section{Preliminaries}
\label{prel}

In what follows we write $U(x)\lesssim V(x)$ if 
there is a constant $C$ such that $U(x)\leq CV(x)$ holds for all $x$ 
in the set in question. We write $U(x)\asymp V(x)$ if both $U(x)\lesssim V(x)$ and
$V(x)\lesssim U(x)$. The standard Landau notations
$O$ and $o$ also will be used.

For the basic notions (such as order  and  type) of entire function theory see, e.g., \cite{lev1, lev}.
The order of an entire function $f$ will be denoted 
by $\rho(f)$ and its zero set by $\mathcal{Z}_f$. 
We denote by $D(z,R)$ the disc with center $z$ of radius $R$. 
The symbol $m_2$ will denote the area Lebesgue measure in $\CC$, while, for a measurable
set $E\subset \RR$, we denote its one-dimensional Lebesgue measure by $|E|$.

\subsection{Functions of bounded type.} In this subsection 
we recall some definitions and basic facts about
functions of bounded type.

Denote by $H^p= H^p(\cp)$, $1\le p\le \infty$, 
the standard Hardy spaces in the upper half-plane. 
For the inner-outer factorization and other basic properties of
the Hardy spaces see, e.g., \cite{ko}. Recall that if $m$ is a non-negative 
function on  $\RR$ such that $\log m \in L^1\big(\frac{dt}{t^2+1}\big)$, then we can define 
the {\it outer function} $O_m$ as
$$
O_m(z) = \exp\bigg(\frac{1}{2\pi i} \int_\RR \bigg(
\frac{1}{t-z}- \frac{t}{t^2+1}\bigg)\log m(t)\, dt \bigg).
$$
We will use the following well-known estimates for outer functions. 
By a very rough estimate $\frac{y}{(t-x)^2+y^2} \lesssim 
\big(y+ \frac{x^2+1}{y}\big)\frac{1}{t^2+1}$
we have for $z =x+iy = r e^{i\theta}\in \CC^+$, $r\ge 1$,
\begin{equation}
\label{out}
\big| \log|O_m(z)| \big| \le \frac{y}{\pi}\int_\RR \frac{|\log m(t)|}{(t-x)^2+y^2} \lesssim
y+ \frac{x^2+1}{y}\lesssim \frac{r}{\sin\theta}. 
\end{equation}                        
In particular, for any $\delta>0$,           
\begin{equation}
\label{out1}
\big|\log|O_m(z)|\big| \lesssim |z|, \qquad \delta<\arg  z<\pi-\delta, \ \ |z|\ge 1. 
\end{equation}

A function $f$ analytic in $\CC^+$ is said to be of {\it bounded type},
if $f=g/h$ for some functions $g$, $h\in H^\infty$. If,
moreover, $h$ can be taken to be outer, we say that $f$ is in
the \textit{Smirnov class $\mathcal{N}_+ = \mathcal{N}_+(\CC^+)$}. 
Analogously, we can define functions of bounded type and Smirnov class functions 
in any given half-plane.

If $f$ is a function of bounded type 
in $\CC^+$, it has the canonical factorization $f = OB S_1/S_2$, 
where $O$ is the outer factor for $f$, $B$ is a Blaschke product and $S_1, S_2$ 
are some (mutually prime) singular inner functions. We define the mean type of $f$ as
$$
{\rm mt}(f) = \limsup\limits_{y\to \infty}\frac{\log|f(iy)|}{y}.
$$
The mean type is equal to $a$ if and only if there is a  
factor $e^{-iaz}$ in the canonical factorization of $f$. If we assume additionally that
$f$ is continuous up to $\mathbb{R}$ then the singular inner functions can not have 
singularities on $\RR$ and so $S_1/S_2= e^{iaz}$ for some $a\in\RR$. Thus, in this case
$f \in \mathcal{N}_+(\CC^+)$ if and only if ${\rm mt}(f) \le 0$. 

Estimate \eqref{out1}, a similar estimate for the singular factor
and the Hayman theorem  \cite[Lecture 15]{lev}
which gives an estimate from below for the Blaschke product outside a union 
of angles of arbitrarily small total size imply the following estimates: 

\begin{lemma}
\label{boun0}
If $f$ is a function of bounded type in $\CC^+$ and ${\rm mt}(f) = a$, 
then, for any $\vep, \delta>0$, there exists $R>0$ such that
$$
\log |f(z)| \le (a \sin \delta + \vep )|z|, \qquad 
\delta<\arg  z<\pi-\delta, \ \ |z|\ge R. 
$$
More generally, if $f = O\frac{B_1 S_1}{B_2 S_2}e^{iaz}$ where
$O$ is the outer factor, $B_1$ and $B_2$ Blaschke products, $S_1, S_2$  
singular inner functions and $a = {\rm mt}(f)$, then
for any $\vep, \delta>0$, there exist $R$ and a set $E\subset [\delta, \pi-\delta]$ 
which is a union of intervals of total length less than $\vep$ such that
$$
(a \sin \delta - \vep )|z| \le \log |f(z)| \le (a \sin \delta + \vep ) |z|, \qquad 
\arg  z \notin E, \ \ |z|\ge R. 
$$
\end{lemma}

For the theory of the Cartwright class we refer to~\cite{hj, ko1, lev}.

The following lemma will be often useful. 

\begin{lemma}
\label{boun}
Let $\mathbb{H}_+$ and $\mathbb{H}_-$ be two complement half-planes and assume that 
$T\subset \mathbb{H}_-$. Then any regularized Cauchy transform $\mathcal{C}_{a, k}$
given by \eqref{reg} is a function from the Smirnov class in $\mathbb{H}_+$.
\end{lemma}

\begin{proof}
Without loss of generality, let $\mathbb{H}_+ = \cp$. 
It is well known that if $f$ is analytic in $\CC^+$ and $\ima f>0$, then $
f$ is in the Smirnov class~\cite[Part 2, Ch. 1, Sect. 5]{hj}. 
Thus, if $u_n>0$ and $\sum_n u_n<\infty$, then  the
function $\sum_n \frac{u_n}{t_n-z}$ is in the Smirnov class $\mathcal{N}_+$.
Consequently, $\sum_n \frac{v_n}{t_n-z} \in \mathcal{N}_+$ 
for any $\{v_n\}\in \ell^1$. 
Finally, $f(z) = z$ also is in $\mathcal{N}_+$ and the result 
follows immediately from formula \eqref{reg}.
\end{proof}


\subsection{Estimates of Cauchy transform in the complex plane.}
The following two results from \cite{bbb-fock}
about the asymptotic behaviour of Cauchy transforms
of measures in the plane will be useful. 
We say that $\Omega\subset \CC$ is a {\it 
set of zero area density} if 
$$
\lim_{R\to\infty} \frac{m_2(\Omega \cap D(0, R))}{R^2} = 0.
$$

\begin{lemma} \cite[Proof of Lemma 4.3]{bbb-fock}
\label{verd}
Let $\nu$ be a finite complex Borel measure
in $\CC$. Then, for any $\varepsilon>0$, 
there exists a set $\Omega$ of zero area density such that
\begin{equation}
\label{meas}
\bigg|\int_\mathbb{C}\frac{d\nu(\xi)}{z-\xi} - \frac{\nu(\mathbb{C})}{z} \bigg|
< \frac{\varepsilon}{|z|}, \qquad  z\in\CC\setminus\Omega.
\end{equation}
\end{lemma}

The following result from \cite{bbb-fock}, which is due to A. Borichev,  
can be considered as an extension of the classical Liouville theorem. 

\begin{theorem} \cite[Lemma 4.2]{bbb-fock}
\label{dens}
If an entire function $f$ of finite order is bounded on 
$\CC\setminus \Omega$ for some set $\Omega$ of zero area density, 
then $f$ is a constant. 
\end{theorem}

Next we discuss growth properties of functions in the spaces 
$\mathcal{H}(T,A,\mu)$. 

\begin{lemma}
\label{gr1}
Let $A$ be an entire function with the zero set $T$ 
and let $A$ be of order $\rho$. Then for any $\vep>0$, there exists
a set $E\subset (0,\infty)$ of zero linear density \textup(i.e., 
$|E \cap (0, R)| = o(R)$, $R\to\infty$\textup) such that
for any entire function $f$ of the form $f = A \mathcal{C}_{a, k}$,
\begin{equation}
\label{cart}
|f(z)| \lesssim |z|^{\rho+k+1+\vep} |A(z)|, \qquad |z|\notin E.
\end{equation}

In particular, if $A$ is of order $\rho$ and type $\sigma$, then
any element of $\mathcal{H}(T,A,\mu)$  is of order at most $\rho$ 
and of type at most $\sigma$ with respect to this order.
\end{lemma}

\begin{proof}
Let $a=(a_n)$ be such that $\sum_n |t_n|^{-k-1} |a_n|<\infty$.
In view of the representation 
$$
A(z)\mathcal{C}_{a, k} (z)  = A(z)P(z) + A(z) z^{k+1} 
\sum_n \frac{a_n}{t_n^{k+1}(z-t_n)},
$$
where $P$ is a polynomial of degree at most $k$, it suffices to prove 
the statement for $\mathcal{C}_{a}$ with $a\in \ell^1$.

For a fixed sufficiently large $n\in \mathbb{N}$ let $\mathcal{R} = \{z: 2^{n}\le |z| \le 2^{n+1}\}$
and $\mathcal{R}' = \{z: 2^{n-1}\le |z| \le 2^{n+2}\}$. 
Let $\{t_{n_1}, t_{n_2}, \dots, t_{n_p} \} = T\cap \mathcal{R}'$. 
Since $A$ is of order 
$\rho$, we have 
$p\lesssim 2^{(\rho+\vep)n}$.

Let $f= A\sum_{n}\frac{a_n}{z-t_n}$. Then,
for $z\in \mathcal{R}$,
$$
\bigg|\sum_{t_n \notin \mathcal{R}'} \frac{a_n}{z-t_n}\bigg|
\lesssim  \sum_{t_n \notin \mathcal{R}'} \frac{|a_n|}{|t_n|}
\lesssim 1.
$$
By the classical Cartan's lemma \cite[Chapter 1, \S 7]{lev1},
there exist discs $D_j$, $j=1, \dots p$, of radii $r_j$ such that $\sum_{j=1}^p r_j <2$
and 
$$
\min_{z\in \mathcal{R} \setminus \cup_j D_j} {\rm dist}\, 
(z, T\cap\mathcal{R}') \ge \frac{1}{p}.
$$
Hence, for $z\in \mathcal{R} \setminus \cup_j D_j$,
we have 
$$
\bigg|\sum_{t_n \in \mathcal{R}'} \frac{a_n}{z-t_n}\bigg| \lesssim
p  \sum_{t_n \in \mathcal{R}'}  
|a_n| \lesssim p
\lesssim |z|^{\rho+\vep}. 
$$
Now we repeat this procedure for any $n\in \mathbb{N}$ and define $E$ as the set of $r$ 
such that $\{|z|=r\} \cap (\cup D_j) \ne \emptyset$. Then \eqref{cart} holds for any
$z$ with $|z|\notin E$.
\end{proof}

Note that we have the following criterion for the inclusion 
of $f$ into $\mathcal{H}(T,A,\mu)$.

\begin{theorem}
\label{inc}
Let $\mathcal{H}(T,A,\mu)$ be a Cauchy--de Branges space
and let $A$ be of finite order. Then an entire function $f$ 
is in $\mathcal{H}(T,A,\mu)$ if and only if the following three conditions hold:
\begin{enumerate} 
\item [(i)] 
$\sum_n\dfrac{|f(t_n)|^2}{|A'(t_n)|^2 \mu_n} <\infty$\textup;      
\smallskip
\item [(ii)]
there exists a set $E\subset (0,\infty)$ 
of zero linear density and $N>0$ such that $|f(z)| \le |z|^N |A(z)|$,
$|z| \notin E$\textup;
\smallskip
\item [(iii)] there exists a set $\Omega$ of positive area density such that 
$|f(z)| = o(|A(z)|)$, $|z|\to \infty$, $z\in \Omega$.
\end{enumerate}
\end{theorem}

\begin{proof}
The necessity of (i) is obvious since for $f= A\sum_{n}\frac{c_n\mu_n^{1/2}}{z-t_n}
\in \mathcal{H}(T,A,\mu)$ we have $f(t_n) = A'(t_n)c_n \mu_n^{1/2}$ and 
$\{c_n\}\in \ell^2$.
The necessity of (ii) is proved in Lemma \ref{gr1}. Finally,  the representation 
$$
\frac{f(z)}{A(z)} = z \sum_{n}\frac{c_n\mu_n^{1/2}}{t_n(z-t_n)} - 
\sum_{n}\frac{c_n\mu_n^{1/2}}{t_n}
$$
and Lemma \ref{verd} imply  that  $f(z)/A(z) = o(1)$ as $|z|\to\infty$
outside a set of zero density. 

To prove the sufficiency consider the function 
$$
H(z)  = \frac{f(z)}{A(z)} - \sum_{n} \frac{f(t_n)}{A'(t_n)(z-t_n)}
$$
which is well defined by (i) and entire. Condition (ii) and Lemma \ref{gr1}
imply2 that $H$ 
is a polynomial. Finally, note that, by the
same argument as above, $\sum_{n} \frac{f(t_n)}{A'(t_n)(z-t_n)}$ tends to zero 
as $|z|\to\infty$ on some 
set $\Omega_1$ whose complement has zero density. Hence, by (iii), 
$|H(z)|\to 0$, $|z|\to\infty$, $z\in \Omega\cap\Omega_1$. Since the set 
$\Omega\cap\Omega_1$ is obviously unbounded, we conclude that $H\equiv 0$. 
Thus, $f$ has the required representation with $c_n = f(t_n)/(A'(t_n)\mu_n^{1/2})$.
\end{proof}

In many cases one can relax the conditions (ii)--(iii) and require the estimates
on a smaller set.

Note that, for $f\in \mathcal{H}(T,A,\mu)$, 
$$
\|f\|^2_{\mathcal{H}(T,A,\mu)} = \|\{c_n\}\|_{\ell^2}^2 = 
\sum_n \frac{|f(t_n)|^2}{|A'(t_n)|^2 \mu_n}.
$$ 
Thus, the space $\mathcal{H}(T,A,\mu)$
is isometrically embedded into the space $L^2(\nu)$, where 
\begin{equation}
\label{nu}
\nu = \sum_n |A'(t_n)|^{-2} \mu_n^{-1} \delta_{t_n}.
\end{equation}
\bigskip


\section{Proof of Theorem \ref{main1}}
\label{proofth1}

In this section we prove Theorem \ref{main1}. In what follows we assume that 
$F = \mathcal{C}_{a, k}/\mathcal{C}_{b, m}$ and $T$ satisfies one of the conditions
$\mathbf{Z}$, $\mathbf{\Pi}$, or $\mathbf{A_\gamma}$.
\medskip

In the {\bf Case} $\mathbf{Z}$ the result follows directly from Lemma~\ref{gr1}.
Let $A$ be a function of zero exponential type with zero set $T$. 
Then $H  = A \mathcal{C}_{a, k}$ 
is an entire function and by Lemma~\ref{gr1} we have 
$|H(z)|\lesssim |z|^N |A(z)|$ for $|z|$ outside some small exceptional set. 
We conclude that $H$ is of zero exponential type. Hence, if 
$F = \mathcal{C}_{a, k}/\mathcal{C}_{b, m}$, then we can write $F=H_1/H_2$
for two functions $H_1, H_2$ of minimal type. By the standard estimates of 
the minimum of modulus for entire functions \cite[Chapter 1, \S 8]{lev}, $F$ 
is of minimal type.  
\medskip

{\bf Case $\mathbf{\Pi}$.} Without loss of generality, let $T \subset 
\{-h \le \ima z\le h\}$. Then, by Lemma \ref{boun}, $F$ is of bounded type in
the half-planes $\CC^+ +i h$ and $\CC^- -ih$. 

Since $T$ has a finite convergence exponent, there exists an entire function $A$ 
of finite order such that $\mathcal{Z}_A = T$. Then, as above, we can write
$F=H_1/H_2$ where $H_1  = A \mathcal{C}_{a, k}$, $H_2  = A \mathcal{C}_{b, m}$.
By Lemma \ref{gr1}, $\rho(H_j)\le \rho(A)$ whence $\rho(F) \le \rho(A)$.  
Choose $\vep\in (0,1)$ such that 
$\rho(F) < \pi/(2\vep)$. By Lemma \ref{boun0}, there exists $R>0$
such that 
$$
\log|F(z)| \lesssim |z|, \qquad \arg z\in [\vep, \pi-\vep] \cup 
[\pi+\vep, 2\pi -\vep], \ |z|>R.
$$
Since $\rho(F) < \pi/(2\vep)$, we can apply the standard 
Phragm\'en--Lindel\"of  principle  to the angles 
$-\vep <\arg z <\vep$ and $\pi-\vep < \arg z<  \pi+\vep$
to conclude that $F$ is of exponential type.
Since $F$ is of bounded type in $\CC^+ ih$, we have
$\log|F(t+ih)| \in L^1\big(\frac{dt}{t^2+1}\big)$. 
Therefore $F(z+ih)$ is of Cartwright class
and, finally, $F$ is of Cartwright class. 
\medskip

{\bf Case $\mathbf{A_\gamma}$.} Here we follow essentially the method of de Branges 
\cite[Theorem 11]{br}. Put $A(\gamma_1, \gamma_2) = \{z: \gamma_1 < \arg z 
<\gamma_2\}$. Choose $\gamma' \in (\gamma, 1 )$ such that $\rho(A) \le 1/\gamma'$.
Without loss of generality we can assume that 
$T \subset A(\delta, \pi\gamma +\delta)$ where $\delta$ is so small that 
  $\pi\gamma +\delta <\pi \gamma'$. 

By Lemma \ref{boun}, $F$ is of bounded type in the half-planes 
$\{-\pi +\delta <\arg z<\delta\}$
and $\{ \pi\gamma +\delta <\arg z<\pi\gamma +\delta +\pi\}$. Then, by Lemma \ref{boun0}, 
we have
$$
\log|F(z)| \lesssim |z|, \qquad \pi\gamma' \le 
\arg z \le 2\pi +\delta/2.
$$
It remains to estimate $|F|$ in the angle $A(\delta/2, \pi\gamma')$.

By Lemma \ref{boun}, $F$ is of bounded type in $\mathbb{C}^-$. Then $\log|F| 
\in L^1\big(\frac{dt}{t^2+1}\big)$. Let $G$ be an outer function in $\CC^+$
such that $|G| = |F|^{-1}$ on $\RR$. Since
$F$ is of bounded type in the half-plane $\{\pi\gamma+\delta <\arg z < 
\pi\gamma+\delta +\pi\}$, we have 
$$
\limsup_{r\to\infty} \frac{\log|F(r e^{i\gamma'})|}{r}  <\infty.
$$
Choose sufficiently large $h>0$ so that $\widetilde{F} = FGe^{ihz}$ is bounded on
the ray $\{\arg z = \gamma'\}$. Also, $\widetilde{F}$ is bounded on $\RR$. 
Let us show that $\widetilde{F}$ is bounded in $A(0, \pi\gamma')$. By Lemma \ref{gr1}, 
$\rho(F) \le \rho(A)$. Choose $\vep>0$ such that $\rho(F) +\vep < \frac{1}{\gamma'}$. 
Then we have
$$
\log|\widetilde{F}(z)| \lesssim |z| + |z|^{\rho(F) +\vep}  +\log|G(z)|, \qquad z\in
A(0, \pi\gamma'), \ \ |z|\ge 1.
$$               

Consider the function $F_1(z) = \widetilde{F} (z^{\gamma'})$. Then $F_1$ is analytic in 
$\CC^+$, continuous up to $\RR$,  and bounded on $\RR$. Using the estimate \eqref{out}
we get
$$
\log|F_1(z)| \lesssim r^{\gamma'} + r^{\gamma'(\rho(F) +\vep)}  
+ \frac{r^{\gamma'}}{\sin \gamma'\theta}, \qquad z=re^{i\theta} \in \CC^+, \ \ r\ge 1.
$$           
We conclude that 
$$
\lim_{r\to\infty} \frac{1}{r} \int_0^\pi \log^+|F_1(r^{i\theta})| \sin \theta\, 
d\theta =0.
$$
By the de Branges version of the Phragm\'en--Lindel\"of principle 
\cite[Theorem 1]{br}, $F_1$ is bounded 
in $\CC^+$. Thus, $\widetilde{F}$ is bounded in $A(0, \pi\gamma')$. 
Using the fact that $|\log|G(z)|| \lesssim |z|$, $z\in A(\delta/2, \pi-\delta/2)$,
we conclude that $\log|F(z)| \lesssim |z|$ for $z\in A(\delta/2, \pi\gamma')$,  
$|z|\ge 1$. Thus, $F$ is of finite exponential type. 

It remains to show the $F$ is of zero type.
Since $\log|F| \in L^1\big(\frac{dt}{t^2+1}\big)$, $F$ is of Cartwright class
and so  we have $\liminf_{|x| \to \infty} 
\frac{\log|F(x)|}{|x|} \le 0$. Hence, by Lemma \ref{boun0}, 
$F$ is of non-positive mean type in the half-planes 
$\{\pi\gamma+\delta < \arg z < \pi\gamma +\delta + \pi\}$ and 
$\{\pi+\delta < \arg z< 2\pi +\delta\}$. Therefore, for any $\vep>0$,
$\log|F(z)| \le \vep |z|$ when $\pi\gamma' < \arg z < 2\pi +\delta/2$
and $|z|$ is sufficiently large. The standard Phragm\'en--Lindel\"of  
principle now implies that $F$ is of zero exponential type.
\qed
\medskip

Given $T$, denote by $\mathcal{C}$ the class of all regularized Cauchy transforms
$\mathcal{C}_{a, k}$ with poles on $T$, by  $\mathcal{C}/\mathcal{C}$ the
class of functions of the form $\mathcal{C}_{a, k}/\mathcal{C}_{b, m}$, etc. Then,
by the same arguments as above one easily obtains the following result that we
will use in what follows: 

\begin{corollary}
\label{wer}
Let $T$ satisfy one of the conditions $\mathbf{Z}$, $\mathbf{\Pi}$ 
or $\mathbf{A_\gamma}$. 
If  $F\in \mathcal{C} + \mathcal{C} \cdot \mathcal{C}/\mathcal{C}$, then 
the conclusions of Theorem \ref{main1} hold.
\end{corollary}

\begin{remark}
{\rm We do not know whether in the case $\mathbf{\Pi}$ the condition that
$T$ has finite convergence exponent can be omitted.   }
\end{remark}
\bigskip


\section{Counterexamples to Krein-type theorem}
\label{count0}

In this section we prove Theorem \ref{count_a}. However, we start with a simpler example
for the special case of the half-plane. Namely, we show that 
there exists a function $F$ with zeros in the lower half-plane which admits 
the representation \eqref{krein} and $F$ is of order 1, but of 
{\it maximal} (i.e., infinite) type. This shows the sharpness of our results
in case $\mathbf{A_\gamma}$ in the limit case $\gamma = 1$. This example will play
an important role in the construction in Theorem \ref{main4}.

\begin{example} 
\label{ex1}
There exists an entire function $F$ of order 1 
with simple zeros $t_n$ in $\CC^-$ such that 
$\sum_n |F'(t_n)|^{-1} <\infty$ and
\begin{equation}
\label{kre}
\frac{1}{F(z)} = \sum_n \frac{1}{F'(t_n)(z-t_n)},
\end{equation}
but $F$ is of maximal type with respect to order 1.

{\rm Let $n_k$ be an increasing sequence such that $n_{k+1}-n_k \ge 1$. Put
$$
G(z) = \prod_{k=1}^\infty \bigg(1-\frac{e^{2\pi iz}}{e^{2\pi n_k}}\bigg).
$$
Then $G$ is an entire function with zeros $z_{m,k} = m-in_k$, $m\in\mathbb{Z}$, $k\in
\mathbb{N}$. By simple estimates of lacunary canonical products, for any $N>0$
there exists $C>0$ such that
$$
\prod_{k=1}^\infty \bigg| 1 -\frac{w}{e^{2\pi n_k}}\bigg| \ge C|w|^N
{\rm dist}\, (w, \{e^{2\pi n_k}\}).
$$
Hence, $|G(z)| \asymp 1$, $\ima z\ge 0$, 
$|G(z)| \gtrsim {\rm dist}\,(z, \{ m-in_k \}_{m,k})$ and so
$|G'(m-in_k)|\gtrsim 1$. Put $F=PG$ where $P$ is a polynomial of degree at least 3
whose zeros are not in the set $\{m-in_k\}$. 
Then $\sum_{t_n \in \mathcal{Z}_F} |F'(t_n)|^{-1} <\infty$
since $\sum_{m, k} |m-in_k|^{-3} <\infty$. Let us show that 
the entire function
$$
H(z) = \frac{1}{F(z)} - \sum_n \frac{1}{F'(t_n)(z-t_n)}
$$ 
is identically zero. From the estimates on $G$ it follows that $|H(z)|\lesssim 1$ 
when ${\rm dist}\, (z,\{t_n\}) \ge 1/2$, and $H(iy)\to 0$,
$y\to\infty$. Hence, $H\equiv 0$.   

If $n_k = k$, the function $F$ is of order 2. However, taking $n_k$ to be sufficiently 
sparse (e.g., $n_k = 2^k$) we obtain an example of a function of order 1 and maximal 
type which has expansion \eqref{kre}.}
\end{example}
\medskip

Now we pass to the proof of Theorem \ref{count_a}. In what follows let 
$D = A(0, \pi\gamma)$ be the angle of size 
$\pi\gamma$ and $\Gamma = \partial D$ be its boundary 
(oriented from $e^{i\pi\gamma} \infty$ to $+\infty$).  

\begin{lemma}
\label{bhh}
Let $g$ be a function analytic in a slightly larger angle $A(-\vep, \pi\gamma +\vep)$
for some $\vep>0$ and assume that, for some $C>0$,
we have
\begin{equation}
\label{kre1}
|g(z)| + |g'(z)| \le \frac{C}{1+|z|^4},
\qquad {\rm dist}\, (z, \Gamma) \le (|z|+1)^{-1}.
\end{equation}
Then for the Cauchy integral of $g$ over $\Gamma$ we have
$$
\int_\Gamma \frac{g(w)}{z-w}dw  = \frac{1}{z}
\int_\Gamma g(w) dw +o\bigg(\frac{1}{z}\bigg), \qquad |z|\to\infty.
$$
\end{lemma}                          

\begin{proof} 
We split the integral into three parts:
$$
\begin{aligned}
\int_\Gamma \frac{g(w)}{z-w}dw  & - \frac{1}{z} 
\int_\Gamma g(w) dw  \\
& 
= \int_{\{|w|<|z|/2\}} \frac{wg(w)}{z(z-w)}dw  + 
\int_{\{|w -z|<(|z|+1)^{-1}\}} \frac{wg(w)}{z(z-w)}dw \\
& + \int_{\{|w| \ge |z|/2,  |w-z|\ge (|z|+1)^{-1} \}} \frac{wg(w)}{z(z-w)}dw 
 = I_1 +I_2 +I_3.
\end{aligned}
$$
Clearly, $|I_1| \lesssim |z|^{-2}$, $|z|\ge 2$, and 
$$
|I_3| \lesssim   
\int_{ \{|w| \ge |z|/2\} } |wg(w)|\,|dw| \lesssim |z|^{-2}.
$$
Finally, to estimate the integral $I_2$ for $z$ close to $\Gamma$ note that
$$
\begin{aligned}
\bigg| \int_{\{|w -z|<(|z|+1)^{-1}\}} \frac{g(w)}{w-z} dw \bigg|
& \le
\bigg| \int_{\{|w -z|<(|z|+1)^{-1}\}} \frac{g(z)}{w-z} dw \bigg| 
\\
& +\bigg| \int_{\{|w -z|<(|z|+1)^{-1}\}} \frac{g(w)-g(z)}{w-z} dw \bigg| \\
& \lesssim
|g(z)| +\max_{|\zeta - z|\le (|z|+1)^{-1}}|g'(\zeta)| \lesssim |z|^{-4}.
\end{aligned}
$$
Combining these inequalities we obtain the estimate of the lemma.
\end{proof}

\begin{proof}[Proof of Theorem \ref{count_a}]
Let $D$ and $\Gamma$ be as above.
Let $f(z) = z^{-4}\sin^4 z$. Then $f$ is an entire function 
of finite exponential type. Put $g(z) = f(z^{1/\gamma})$, $z\in D$,  
and define the function $F$ for 
$z\in \mathbb{C} \setminus \overline{D}$ by the contour integral 
$$
F(z) = \frac{1}{2\pi i}\int_\Gamma \frac{g(w)}{z-w} dw, \qquad 
z\in \mathbb{C} \setminus \overline{D}
$$
(note that $|g(w)|\lesssim |w|^{-4/\gamma}$ and so there 
is no problem with convergence). 

Now we use a well-known trick to show that $F$ admits a continuation to an entire
function. Let $R>0$ and $\widetilde{\Gamma}$ be the contour 
$\{t e^{i\pi\gamma}: t\ge R\} \cup \{r e^{i\theta}: 0\le\theta\le \pi\gamma \}
\cup \{t\ge R\}$ 
(also oriented from $e^{i\pi\gamma} \infty$ to $+\infty$) 
and let $\widetilde{D}\subset D$ be the domain
such that $\partial \widetilde{D} = \widetilde{\Gamma}$.
Put
$$
\widetilde{F} (z) = \frac{1}{2\pi i}
\int_{\widetilde{\Gamma}} \frac{g(w)}{z-w} dw, \qquad 
z\in \mathbb{C} \setminus \overline{\widetilde{D}}.
$$
Then, for $z \in \mathbb{C} \setminus \overline{D}$,
$$
F(z) - \widetilde{F} (z) = \frac{1}{2\pi i}
\int_{\Gamma_0} \frac{g(w)}{z-w} dw = 0
$$
where $\Gamma_0$ is the counterclockwise oriented boundary of the sector
$S = \{r e^{i\theta}: 0< r< R, 0<\theta<\pi\gamma\}$ and the integral is zero since
$g$ is analytic inside $S$ and continuous up to the boundary. Thus, $\widetilde{F}$ 
is a continuation of $F$ to a larger domain 
$\mathbb{C} \setminus \overline{\widetilde{D}}$. Since $R$ is arbitrary, we conclude
that $F$ has an entire extension. Moreover, by the same argument 
we have a representation for $F$ inside $D$: 
$$
F(z) = \frac{1}{2\pi i}\int_\Gamma \frac{g(w)}{z-w} dw +g(z), \qquad z\in D.
$$

Note that $g$ satisfies the hypotheses of Lemma \ref{bhh}. Indeed, 
$|f(z)|+|f'(z)| \lesssim (|z|+1)^{-4}$ for ${\rm dist}\,(z, \mathbb{R})\ge 1$
which implies estimate \eqref{kre1}. Also, since $f$ is even and non-negative, 
it follows that $\alpha := (2\pi i)^{-1} 
\int_\Gamma g(w)dw \ne 0$. Thus, 
\begin{equation}
\label{f1}
F(z) = \alpha z^{-1} +o(z), 
\qquad z\in \mathbb{C} \setminus D, 
\end{equation}
and we conclude that $F$ has at most 
finite number of zeros in $\mathbb{C} \setminus D$. 
Let us analyse the zeros of $F$ inside $D$. We have
\begin{equation}
\label{f2}
F(z) = g(z)+\frac{\alpha}{z} + o\Big(\frac{1}{z}\Big), \qquad z\in D, 
\ |z|\to\infty.
\end{equation}
Equivalently, this means that for $G(z) = F(z^\gamma)$ we have
$$
G(z) =\frac{\sin^4 z}{z^4}+\frac{\alpha}{z^\gamma} + 
o\Big(\frac{1}{|z|^\gamma}\Big), \qquad z\in \mathbb{C}^+.
$$
The unperturbed equation
$$
\frac{\sin^4 z}{z^4}+\frac{\alpha}{z^\gamma} = 0
$$
has zeros in $\mathbb{C}^+$ whose asymptotics can be easily computed. Namely, 
if we write $-8\alpha = re^{i\beta}$, then 
the solutions $z_k=x_k+iy_k$
of the equation $(e^{-iz} - e^{iz})^4 = re^{i\beta}z^{4-\gamma}$ in 
$\mathbb{C}^+$ will have the asymptotics
$$
\begin{cases}
x_k = \frac{\pi k}{2} +\frac{\beta}{4} +o(1),\\
y_k = \big(1-\frac{\gamma}{4}\big)\ln k +\big(1-\frac{\gamma}{4}\big) 
\ln\frac{\pi}{2} +o(1),
\end{cases}
$$
as $k\to\infty$. It is easy to see that 
$$
\bigg|\frac{\sin^4 z}{z^4}+\frac{\alpha}{z^\gamma}\bigg| \gtrsim \frac{1}{|z|^\gamma} 
$$ 
when $z\in \mathbb{C}^+$, ${\rm dist}\,(z,\{z_k\})\ge \frac{1}{10}$ and $|z|$ 
is sufficiently large. By the Rouch\'e theorem, for sufficiently large $k$ the disc
$D(z_k, 1/10)$ contains exactly one zero (say, $s_k$) of $G$ and these are
all zeros of $G$ except a finite number.

Moreover, 
$$
|G'(s_k)|\asymp \frac{|\sin^3 s_k\cos s_k|}{|s_k|^4} \asymp
\frac{1}{|s_k|^\gamma}.
$$

It follows from formulas \eqref{f1} and \eqref{f2} that $F$ is an entire function of 
order $\gamma^{-1}$ and of finite type. The zeros of $F$ are given by $t_k = s_k^\gamma$. 
Dividing and mutiplying by a polynomial we can assume without loss of generality 
that all zeros of $F$ are simple and lie in $D$. 
Since $|F'(t_k)|\asymp|s_k|^{1-2\gamma}\gtrsim |t_k|^{-1}
\asymp |k|^{-\gamma}$, we can multiply $F$ by a polynomial $P$ of sufficiently 
large degree (with zeros in $D$) to achieve
$$
\sum_n \frac{1}{|F'(t_n)P(t_n)|} <\infty. 
$$
Slightly abusing the notation we now denote by $\{t_n\}$ the zero set of $FP$.
It remains to show that $FP$ has the required simple fraction expansion. We have
$$
\frac{1}{F(z)P(z)} = \sum_n \frac{1}{(FP)'(t_n) (z-t_n)} +H(z)
$$
for some entire function $H$. Since $FP$ is of order $\gamma^{-1}$,
we conclude that $H$
is of order at most $\gamma^{-1}$. However,  $|F(z)P(z)|\gtrsim 1$, $z\in
\mathbb{C}\setminus D$ (since $|F(z)|\asymp |z|^{-1}$ there). Also, 
for any $\vep>0$ we have 
$|F(z)P(z)|\asymp |g(z)| \to\infty$ when $|z|\to \infty$ and 
$z\in A(\vep, \pi\gamma-\vep)$. From this we conclude that $H\equiv 0$.                                         
\end{proof}
\bigskip


\section{Proof of Theorems \ref{main2} and \ref{main3}}
\label{proofth23}

We first state a simple proposition which shows  
that nearly invariance implies division-invariance.
The proof is similar to \cite[Proposition 5.1]{ar} and we omit it.
Let $\mathcal{H}$ be a reproducing kernel Hilbert space which consists 
of analytic functions in some domain $D$ and has 
the division property. Recall that, for a closed subspace $\mathcal{H}_0$ 
of $\mathcal{H}$, we denote by $\mathcal{Z}(\mathcal{H}_0)$ 
the set of its common zeros.

\begin{proposition}
\label{nea}
Assume that there exists $w_0\in D$ such that
$\frac{f(z)}{z-w_0} \in \mathcal{H}_0$ whenever $f\in \mathcal{H}_0$
and $f(w_0)= 0$. Then, for any $w\in D\setminus \mathcal{Z}(\mathcal{H}_0)$
and any $f\in \mathcal{H}_0$ such that $f(w) =0$, we have
$\frac{f(z)}{z-w} \in \mathcal{H}_0$.
\end{proposition}
\medskip

We pass to the proofs of Theorems \ref{main2} and \ref{main3}.
The key idea of the proof is due to L.~de~Branges \cite[Theorem 35]{br}. 
Assume that neither $\mathcal{H}_1 \subset \mathcal{H}_2$ nor
$\mathcal{H}_2 \subset \mathcal{H}_1$ and choose nonzero functions
$F_1, F_2 \in \mathcal{H}$ such that $F_1 \perp \mathcal{H}_2$ but
$F_1$ is not orthogonal to  $\mathcal{H}_1$, while
$F_2 \perp \mathcal{H}_1$ but $F_2$ is not orthogonal to $\mathcal{H}_2$.

Let $F\in \mathcal{H}_1$ and $G\in \mathcal{H}_2$. Define two functions
$$
\begin{aligned}
f(w) & = \bigg\langle \frac{F - \frac{F(w)}{G(w)}G}{z-w}, F_1\bigg\rangle_{\mathcal{H}(T,A,\mu)} = 
\int\frac{F(z) - \frac{F(w)}{G(w)}G(z)}{z-w} \overline{F_1(z)} d\nu(z), \\
g(w) & = \bigg\langle \frac{G - \frac{G(w)}{F(w)}F}{z-w}, F_2\bigg\rangle_{\mathcal{H}(T,A,\mu)}
= \int\frac{G(z) - \frac{G(w)}{F(w)}F(z)}{z-w} \overline{F_2(z)} d\nu(z),
\end{aligned}
$$
where $\nu$ is the measure defined by \eqref{nu} such that 
the embedding  $\mathcal{H}(T,A,\mu) \subset L^2(\nu)$ is isometric. 
The functions $f$ and $g$ are well-defined and analytic on the sets $\{w: G(w)\ne 0\}$
and $\{w: F(w)\ne 0\}$, respectively.
\medskip
\\
{\bf Step 1:} {\it $f$ and $g$ are entire functions, 
$f$ does not depend on the choice of $G$ and $g$ does not depend on the choice of $F$}.

Let $f_1$ be a function associated in a similar way to $G_1 \in \mathcal{H}_1$,
$$
f_1(w)  = 
\int\frac{F(z) - \frac{F(w)}{G_1(w)}G_1(z)}{z-w} \overline{F_1(z)} d\nu(z).
$$
Then, for $G(w) \ne 0$ and $G_1(w) \ne 0$, we have
$$
f_1(w) - f(w) = \frac{F(w)}{G(w)G_1(w)} 
\int \frac{G_1(w)G(z) - G(w)G_1(z)}{z-w} \overline{F_1(z)} d\nu(z) = 0,
$$
since $\frac{G_1(w)G - G(w)G_1}{z-w} \in \mathcal{H}_2$.

Now choosing $G$ such that $G(w) \ne 0$ we can extend $f$ analytically to 
a neighborhood of the point $w$.
Thus, $f$ and $g$ are entire functions.
\medskip
\\
{\bf Step 2:} {\it $f$ and $g$ are of zero exponential type.}

Recall that we denote by $\mathcal{C}$ the class of all regularized Cauchy transforms
with poles in $T$. Since $F$ and $G$ are in $\mathcal{H}(T,A,\mu)$, we 
have $F/A, G/A \in \mathcal{C}$ and so $F/G, G/F \in \mathcal{C}/\mathcal{C}$. 
Hence, 
$$
f, g \in  \mathcal{C} +\mathcal{C} \cdot \frac{\mathcal{C}}{\mathcal{C}}.
$$
By Corollary \ref{wer} $f$ and $g$ are of zero exponential type.
\medskip
\\
{\bf Step 3:} {\it Either $f$ or $g$ is identically zero.}

Given $w$ such that $F(w)\ne 0$, $G(w) \ne 0$, 
we have
\begin{equation}
\label{ots}
\begin{aligned}
|f(w)| & \le \bigg|\int \frac{F(z)\overline{F_1(z)}}{z-w} d\nu(z) \bigg| + 
\bigg|\frac{F(w)}{G(w)}\bigg|\cdot 
\bigg|\int \frac{G(z)\overline{F_1(z)}}{z-w} d\nu(z) \bigg|,\\
|g(w)| & \le \bigg|\int \frac{G(z)\overline{F_2(z)}}{z-w} d\nu(z) \bigg| + 
\bigg|\frac{G(w)}{F(w)}\bigg|\cdot 
\bigg|\int \frac{F(z)\overline{F_2(z)}}{z-w} d\nu(z)\bigg|.
\end{aligned}
\end{equation}
By Lemma \ref{gr1}, there exist $M>0$ and a set $E\subset (0, \infty)$ of zero 
linear density such that
$$         
|f(w)| \lesssim |w|^M\bigg(1+ \bigg|\frac{F(w)}{G(w)}\bigg|\bigg), \qquad
|g(w)| \lesssim |w|^M\bigg(1+ \bigg|\frac{G(w)}{F(w)}\bigg|\bigg), \qquad |w|\notin E.
$$
We conclude that
$$
\min\big(|f(w)|, |g(w)|\big) \lesssim |w|^M, \qquad |w|\notin E.
$$
Since $E$ has zero linear density, we can choose a sequence 
$R_j\to \infty$ such that $R_j\notin E$ and $R_{j+1}/ R_j \le 2$. Applying 
the maximum principle to the annuli $R_j\le |z|\le R_{j+1}$, we conclude that
$$
\min\big(|f(w)|, |g(w)|\big) \lesssim |w|^M, \qquad |w|\ge 1.
$$
Since both $f$ and $g$ are of zero exponential type, 
a small variation of a 
well-known deep result by de Branges \cite[Lemma 8]{br} gives that either $f$ 
or $g$ is a polynomial. 

Assume that $f$ is a nonzero polynomial. By Lemma \ref{verd}, 
there exists a set $\Omega$ of zero area density such that
$$
\bigg|\int\frac{F(z)\overline{F_1(z)}}{z-w} d\nu(z) \bigg|
+
\bigg|\int\frac{G(z)\overline{F_1(z)}}{z-w} d\nu(z) \bigg| 
=O\Big(\frac{1}{|w|}\Big), \qquad w\notin \Omega. 
$$
Hence, $|F(w)/G(w)| \to\infty$ as $|w|\to\infty$, $w\notin\Omega$, and so
$$
|g(w)| \le \bigg|\int\frac{G(z)\overline{F_2(z)}}{z-w} d\nu(z) \bigg| + 
\bigg|\frac{G(w)}{F(w)}\bigg|\cdot 
\bigg|\int \frac{F(z)\overline{F_2(z)}}{z-w} d\nu(z)\bigg| 
= O\Big(\frac{1}{|w|}\Big), \qquad w\notin \Omega\cup \widetilde{\Omega}, 
$$
where $\widetilde{\Omega}$ is another set of zero area density 
(here we again applied Lemma \ref{verd}). 
Thus, $g$ tends to zero outside a set of zero density
and so $g\equiv 0$ by Theorem \ref{dens}.
\medskip
\\
{\bf Step 4:} {\it End of the proof.} 

Without loss of generality, let 
$f\equiv 0$. Then
$$
\frac{F(w)}{G(w)}\int\frac{G(z)\overline{F_1(z)}}{z-w}  d\nu(z)
=\int\frac{F(z)\overline{F_1(z)}}{z-w}  d\nu(z)
$$                                             
for any $F\in \mathcal{H}_1$, $G\in \mathcal{H}_2$.

Recall that $F_1$ is not orthogonal to $\mathcal{H}_1$ and so we can choose
$F \in \mathcal{H}_1$ such that
$\langle F, F_1\rangle = \int F\overline{F}_1 d\nu \ne 0$. Then, by 
Lemma \ref{verd}, 
$$
\bigg|\int\frac{F(z)\overline{F_1(z)}}{z-w}  d\nu(z)\bigg| \gtrsim \frac{1}{|w|}, 
\qquad w\notin \Omega,
$$
for some set $\Omega$ of zero density. Since $G\perp F_1$ for 
any $G\in \mathcal{H}_2$, 
we have (again by Lemma \ref{verd}) 
$$
\bigg|\int\frac{G(z)\overline{F_1(z)}}{z-w}  d\nu(z)\bigg| = o\Big(\frac{1}{|w|}\Big), 
\qquad |w|\to\infty, \ w\notin \widetilde{\Omega},
$$
where $\widetilde{\Omega}$ is another set of zero density. We conclude that
$|F(w)/G(w)| \to \infty$ when $|w|\to \infty$ 
outside the set of zero density 
$\Omega \cup \widetilde{\Omega}$ (for any $G\in\mathcal{H}_2$). 
Applying this fact and 
Lemma \ref{verd} to $g$ we conclude that $|g(w)|\to 0$ 
outside a set of zero density
and so $g\equiv 0$ by Theorem \ref{dens}.

Thus, we have
\begin{equation}
\label{ghj}
\frac{G(w)}{F(w)}\int\frac{F(z)\overline{F_2(z)}}{z-w}  d\nu(z)
=\int\frac{G(z)\overline{F_2(z)}}{z-w}  d\nu(z)
\end{equation}
and we may repeat the above argument. Choose $G \in \mathcal{H}_2$
such that $\langle G, F_2\rangle = \int G\overline{F}_2 d\nu \ne 0$. Then,
by Lemma \ref{verd}, the modulus of the 
right-hand side in \eqref{ghj} is $\gtrsim |w|^{-1}$,
while the left-hand side is $o(|w|^{-1})$ when $|w|\to\infty$ 
outside a set of zero density. This contradiction proves Theorem \ref{main2}.
\qed
\medskip

\begin{proof}[Proof of Theorem \ref{main3}]
The proof essentially coincides with the proof of Theorem \ref{main2}.
Let $f$ and $g$ be defined as above. By Step 1, $f$ and $g$ are entire functions. 
Since $f, g\in \mathcal{C} +\mathcal{C} \cdot \mathcal{C} / \mathcal{C}$,
we conclude by Corollary \ref{wer} that $f$ and $g$ are of finite exponential type. 

Now we need to show that $f$ and $g$ are of zero type. 
For this we can use the property that 
$\mathcal{H}_1$ and $\mathcal{H}_2$ 
are closed under the $*$-transform. 
Therefore, we can also take $F^*$ and $G^*$ in place of $F$ and $G$
in the definition of $f$ and $g$.  
Since $T\subset \{ -h \le \ima z\le h\}$, we have 
$$
\int\frac{\alpha(z)}{z-w} d\nu(z) \lesssim |\ima w|^{-1}, 
\qquad |\ima w| \ge 2h,
$$
for any $\alpha\in L^1(\nu)$. 
Thus, 
$$
|f(w)| \lesssim \bigg( 1+ \min\bigg\{\bigg|\frac{F(w)}{G(w)}\bigg|,
\bigg|\frac{F(\overline w)}{G(\overline w)}\bigg|
\bigg\} \bigg) \frac{1}{|\ima w|}, \qquad |\ima w| \ge 2h.
$$

Note that  $F/G \in \mathcal{C}/\mathcal{C}$. Therefore, 
$F/G$ is a function of bounded type in $\CC^+ + ih$ 
and in $\CC^- -ih$. Then we can write for $w\in \CC^+$,
$$
\frac{F(w+ih)}{G(w+ih)} = O\frac{B_1 S_1}{B_2 S_2}e^{iaw}, \qquad
\frac{\overline{ F(\overline w -ih)}}{\overline{ G(\overline w - ih)}}  
= \tilde O\frac{\tilde B_1 \tilde S_1}{\tilde B_2 \tilde S_2}e^{ibw},
$$
where
$O, \tilde O$ are the corresponding outer factors, $B_1, B_2, \tilde B_1, \tilde B_2$ 
are Blaschke products in $\CC^+$, $S_1, S_2, \tilde S_1, \tilde S_2$ are 
singular inner functions in $\CC^+$ (without factors of the form $e^{icz}$) 
and $a, b\in \RR$. If at least one of the numbers $a$ 
or $b$ is non-negative we have, by Lemma \ref{boun0}, 
$$
\log |f(w)| =o(|w|), \qquad |\ima w|\ge 2h, \ \arg w\notin E,
$$
where $E\subset [0, 2\pi]$ is a union of interval of arbitrarily small 
total length. Since $f$ is an entire function of exponential type, the
classical Phragm\'en--Lindel\"of principle implies that $f$ is of zero type. 
Assume that both $a<0$ and $b<0$. Then for $g$ we have a similar estimate
$$
|g(w)| \lesssim \bigg( 1+ \min\bigg\{\bigg|\frac{G(w)}{F(w)}\bigg|,
\bigg|\frac{G(\overline w)}{F(\overline w)}\bigg|
\bigg\} \bigg) \frac{1}{|\ima w|}, \qquad |\ima w| \ge 2h.
$$
We conclude from factorizations of 
$G/F$ in $\CC^+ + ih$ and in $\CC^- -ih$ that $g$ tends to 
zero outside the union of angles 
of arbitrarily small total size whence $g\equiv 0$. 

Thus, we have seen that
\begin{enumerate} 
\item [(i)]
either both $f$ and $g$ are of zero type, and we can proceed to Step 3 as in the proof
of Theorem \ref{main2}; 
\item [(ii)]
or one of the functions $f$ or $g$ is identically zero, and we can go directly to Step 4.
\end{enumerate}

The end of the proof is the same as in Theorem \ref{main2}.
\end{proof}

\begin{remark}
{\rm Let us mention that Lemma \ref{verd} and 
Theorem \ref{dens} are essential in the case
$\mathbf{Z}$. In the cases $\mathbf{\Pi}$ and $\mathbf{A_\gamma}$ 
it is sufficient to consider the asymptotics of the Cauchy transforms
along the rays lying outside the strip or the angle in question.  }
\end{remark}
                                    
\begin{remark}     
{\rm  Theorems \ref{main2} and \ref{main3} can be extended 
with essentially the same proofs to the case of nearly invariant subspaces
having the same sets of common zeros (counting with multiplicities).}
\end{remark}
\bigskip


\section{Counterexample to the Ordering Theorem}
\label{count}

In this section we prove Theorem \ref{main4}. Our construction is similar 
to Example \ref{ex1}. To have the symmetry with respect to $\RR$ we
rotate the function $G$ from Example \ref{ex1}. 

\begin{proof}[Proof of Theorem \ref{main4}] {\bf Step 1.} 
Define the entire functions $A_1$ and $A_2$ by
$$
A_1(z) = \prod_{k=1}^\infty \bigg(1-\frac{e^{2\pi z}}{e^{2\pi k}}\bigg), \qquad 
A_2(z) = \prod_{k=1}^\infty \bigg(1-\frac{e^{-2\pi z}}{e^{2\pi k}}\bigg),
$$
and put $A=A_1A_2$. Then the zero set of $A$ is given by
$T = \{t_n\} = \{k+im: k\in \mathbb{Z}\setminus\{0\}, m\in \mathbb{Z} \}$.
Put $\mu_n = |t_n|^{-3}$. 

Now we construct the functions $G_1$ and $G_2$ as follows. Let $P_1$ be some polynomial 
of degree 4 such that $\mathcal{Z}_{P_1} \subset \mathcal{Z}_{A_2}$ 
and let $G_1$ be an entire function with simple zeros 
$\tilde t_n = t_n+\delta_n$, $t_n\in \mathcal{Z}_{A_2} \setminus \mathcal{Z}_{P_1}$, 
$|\delta_n|<1/100$,
such that: 
\begin{enumerate} 
\item [(i)]
$\mathcal{Z}_{G_1}$ is so close to the set 
$\mathcal{Z}_{A_2} \setminus \mathcal{Z}_{P_1}$ that
\begin{equation}
\label{tre}
|G_1(z)P_1(z)| \asymp |A_2(z)|, \qquad {\rm dist}\, (z, \mathcal{Z}_{A_2})>1/10; 
\end{equation}
\item [(ii)]
$\mathcal{Z}_{G_1} \cap T =\emptyset$.
\end{enumerate}
Note that condition \eqref{tre} implies (by the maximum modulus 
principle applied to the function $(z-t_n)G_1P_1/A_2$ in $\{|z-t_n|\le 1/10\}$) that
\begin{equation}
\label{tre1}
|G_1(t_n)P_1(t_n)| \lesssim |A_2'(t_n)|,  \qquad t_n \in \mathcal{Z}_{A_2}.
\end{equation}

Analogously, we define a polynomial $P_2$ and the function $G_2$ such that
$|G_2(z)P_2(z)| \asymp |A_1(z)|$ when ${\rm dist}\, (z, \mathcal{Z}_{A_1})>1/10$. 
\medskip
\\
{\bf Step 2.} Let us show that $G_1, G_2$ belong to the corresponding space  
$\mathcal{H}(T,A,\mu)$. Similarly to Example \ref{ex1} we have
$$
\begin{aligned}
|A_1(z)| & \asymp 1, \qquad \rea z \le 0, \\
|A_1'(t_n)| & \gtrsim 1, \qquad t_n \in \mathcal{Z}_{A_1}, \\
|A_1(z)| & \gtrsim 1, \qquad {\rm dist}\, (z, \mathcal{Z}_{A_1})>1/10.
\end{aligned}
$$ 
Hence, 
$$
\bigg|\frac{G_1(z)}{A(z)}\bigg| \lesssim \bigg|\frac{1}{P_1(z)A_1(z)}\bigg|
\lesssim \frac{1}{|P_1(z)|}, 
\qquad {\rm dist}\, (z, T)>1/10.
$$
Also, by \eqref{tre} and \eqref{tre1},
$$
\begin{aligned}
\sum_{t_n\in T\setminus \mathcal{Z}_{P_1}} \frac{|G_1(t_n)|^2}{|A'(t_n)|^2\mu_n} & \asymp
\sum_{t_n \in \mathcal{Z}_{A_1}}  \frac{|A_2(t_n)|^2}{|P_1(t_n)|^2
|A_2(t_n)|^2|A_1'(t_n)|^2\mu_n}  +
\sum_{t_n \in \mathcal{Z}_{A_2} \setminus \mathcal{Z}_{P_1}} 
\frac{|G_1(t_n)|^2}{|A_1(t_n)|^2|A_2'(t_n)|^2\mu_n} \\
& \lesssim 
\sum_{t_n \in \mathcal{Z}_{A_1}}  \frac{1}{|P_1(t_n)|^2
|A_1'(t_n)|^2\mu_n} + 
\sum_{t_n \in \mathcal{Z}_{A_2} \setminus \mathcal{Z}_{P_1}} 
\frac{1}{|P_1(t_n)|^2|A_1(t_n)|^2\mu_n}. 
\end{aligned}
$$
Since $\mu_n = |t_n|^{-3}$ and $|P_1(t_n)| \asymp |t_n|^4$, 
$t_n\in T\setminus \mathcal{Z}_{P_1}$, we conclude that
$\sum_{t_n\in T} |G_1(t_n)|^2/(|A'(t_n)|^2\mu_n) < \infty$. By Theorem \ref{inc}, 
$G_1 \in \mathcal{H}(T,A,\mu)$.
                                                                           
Clearly, we can construct $G_1$ and $G_2$ so that $G_1^* = G_1$, $G_2^* = G_2$,
whence the spaces $\mathcal{H}_{G_1}$ and $\mathcal{H}_{G_2}$ are $*$-closed.
\medskip
\\
{\bf Step 3.} Now we construct a function $f \in \mathcal{H}(T,A,\mu)$ such that $f\perp \mathcal{H}_{G_1}$, but 
$\langle\frac{G_2}{z-\lambda}, f\rangle \ne 0$ for some $\lambda\in \mathcal{Z}_{G_2}$. Thus, 
$\mathcal{H}_{G_2}$ is not contained in $\mathcal{H}_{G_1}$. By the symmetry 
of the construction, also $\mathcal{H}_{G_1}$ is not contained in $\mathcal{H}_{G_2}$.

Since $|A_1(z)| \asymp 1$, $\rea z \le 0$,
we have $\sum_n |A_1(t_n)|^2 |t_n|^{-3} <\infty$. Put
$$
f(z) = A (z)\sum_n \frac{\overline{A_1(t_n)}}{|t_n|^3 (z-t_n)} = 
A(z) \sum_n \frac{\overline{A_1(t_n)}\mu_n^{1/2}}{|t_n|^{3/2} (z-t_n)}.
$$
Then $f\in \mathcal{H}(T,A,\mu)$. For $\lambda\in \mathcal{Z}_{G_1}$ 
we have
$$
\bigg\langle \frac{G_1}{z-\lambda}, f \bigg\rangle = 
\sum_n \frac{G_1(t_n)}{(t_n-\lambda)A'(t_n)\mu_n^{1/2}}\cdot \frac{A_1(t_n)}{|t_n|^{3/2}} = 
\sum_{t_n \in \mathcal{Z}_{A_2}}  
\frac{G_1(t_n)}{A_2'(t_n)(t_n-\lambda)}.
$$
To show that the last expression is 0, we prove the interpolation formula
$$
\sum_{t_n \in \mathcal{Z}_{A_2}}    
\frac{G_1(t_n)}{A_2'(t_n)(z- t_n)} = \frac{G_1(z)}{A_2(z)}.
$$
Once it is proved, taking $z=\lambda \in \mathcal{Z}_{G_1}$ we obtain that 
$\langle \frac{G_1}{z-\lambda}, f \rangle$.
As usual, consider the entire function
$$
H(z) = \frac{G_1(z)}{A_2(z)} - 
\sum_{t_n \in \mathcal{Z}_{A_2}}    
\frac{G_1(t_n)}{A_2'(t_n)(z- t_n)}.
$$
By \eqref{tre1}, $|G_1(t_n)/A'(t_n)| \lesssim |P_1(t_n)|^{-1}$
for $t_n \in \mathcal{Z}_{A_2} \setminus
\mathcal{Z}_{P_1}$, and it is easy to see that
$$
\bigg|\sum_{t_n \in \mathcal{Z}_{A_2}}    
\frac{G_1(t_n)}{A_2'(t_n)(z- t_n)} \bigg| \to 0, \qquad |z|\to\infty, \ 
{\rm dist}\, (z, \mathcal{Z}_{A_2})>1/10.
$$
Since, by \eqref{tre}, $|G_1(z)/A_2(z)| \asymp |P_1(z)|^{-1}$
when ${\rm dist}\, (z, \mathcal{Z}_{A_2})>1/10$, 
we conclude that $|H(z)| \to 0$, when $|z|\to\infty$
and ${\rm dist}\, (z, \mathcal{Z}_{A_2})>1/10$. Hence, $H\equiv 0$.
\medskip
\\
{\bf Step 4.} It remains to show that 
$\langle\frac{G_2}{z-\lambda}, f\rangle \ne 0$ for some $\lambda\in \mathcal{Z}_{G_2}$.
As above we have 
$$
\bigg\langle \frac{G_2}{z-\lambda}, f \bigg\rangle = 
\sum_n \frac{G_2(t_n)}{(t_n-\lambda)A'(t_n)\mu_n^{1/2}}\cdot \frac{A_1(t_n)}{|t_n|^{3/2}} = 
\sum_{t_n \in \mathcal{Z}_{A_2}}  
\frac{G_2(t_n)}{A_2'(t_n)(t_n-\lambda)}.
$$
Assume that $\langle\frac{G_2}{z-\lambda}, f\rangle = 0$ 
for any $\lambda\in \mathcal{Z}_{G_2}$. Then the entire function 
$A_2(z) \sum_{t_n \in \mathcal{Z}_{A_2}}  
\frac{G_2(t_n)}{A_2'(t_n)(z-t_n)}$ vanishes on $\mathcal{Z}_{G_2}$ and we can write
$$
A_2(z) \sum_{t_n \in \mathcal{Z}_{A_2}}  
\frac{G_2(t_n)}{A_2'(t_n)(z-t_n)} = G_2(z)U(z)
$$
for some entire function $U$. Since $G_2\ne 0$ on $T$, comparing the values at 
$t_n\in \mathcal{Z}_{A_2}$, we obtain $U(t_n) = 1$,  $t_n \in \mathcal{Z}_{A_2}$.
Hence, we can write $U = 1+ A_2V$ for some entire function $V$. Dividing by $G_2A_2$
we get
$$
V(z) = \frac{1}{G_2(z)}
\sum_{t_n \in \mathcal{Z}_{A_2}} \frac{G_2(t_n)}{A_2'(t_n)(z-t_n)} - \frac{1}{A_2(z)}.
$$ 
We know that $|A_2(z)| \gtrsim 1$ and $|G_2(z)| \gtrsim |P_2(z)|^{-1}$ when 
${\rm dist}\, (z, T)>1/10$. 
Since $|G_2(t_n)| \asymp |A_1(t_n)|/|P_2(t_n)| \asymp 1/|P_2(t_n)|$,
$t_n \in \mathcal{Z}_{A_2}$, we also see that the Cauchy transform in the above formula
tends to zero when $|z|\to\infty$, $
{\rm dist}\, (z, T)>1/10$. We conclude that $V$ is at most a polynomial. However, 
we also have $\limsup_{x\to+\infty}|G_2(x)|=\infty$, whence $V$ is at most a constant.
Since $\lim_{x\to+\infty} A_2(x) = 1$, we conclude that this constant is $-1$.

We have arrived to the following representation:
\begin{equation}
\label{wrt}
\sum_{t_n \in \mathcal{Z}_{A_2}}  
\frac{G_2(t_n)}{A_2'(t_n)(z-t_n)} = \frac{G_2(z)}{A_2(z)} - G_2(z)
\end{equation}                                      
and we need to show that it is impossible. Note that \eqref{wrt} 
implies that 
$$
|1-A_2(x)|\lesssim |G_2(x)|^{-1}, \qquad x >0.
$$
However, 
$$
1-A_2(x) = 1-
\prod_{k=1}^\infty \big( 1- e^{-2\pi x -2\pi k}\big) > 1-
\big( 1- e^{-2\pi x -2\pi }\big) = e^{-2\pi x -2\pi }.
$$
On the other hand, it is clear that 
$$
\frac{\log|G_2(x)|}{x} \to\infty, \qquad  
x\to+\infty,\ \ \dist(x, \mathcal{Z}_{G_2}) >1/10.
$$
Indeed, for the lacunary product $L(z) = \prod_{k=1}^\infty \big( 1- e^{-2\pi k}z\big)$
and for any $N, \delta>0$,   
we have $|L(z)| \gtrsim |z|^N$, ${\rm dist}\, (z,\{e^{2\pi k}\})>\delta$. 
By the construction, $|G_2(x)|\asymp |L(e^{x})|/|P_2(x)|
\gtrsim e^{Nx}/|P_2(x)|$ when $\dist(x, \mathcal{Z}_{G_2}) >1/10$.
This contradiction proves Theorem \ref{main4}.
\end{proof}


\end{document}